\documentclass{amsart}
\usepackage{amssymb,amsmath,amsthm,amsfonts}
\usepackage{framed,comment,enumerate}
\usepackage[pdftex]{graphicx}
\usepackage{epstopdf}
\usepackage{xcolor, framed}
\usepackage{color}
\def\R {\mathbb{R}}
\def\eps{\varepsilon}
\def\MA {Monge-Amp\`ere }
\def\Shu{S_h[u]}

\newtheorem{prop}{Proposition}[section]
\newtheorem{thm}{Theorem}[section]

\newtheorem{lem}{Lemma}[section]
\theoremstyle{definition}
\newtheorem{defi}{Definition}[section]

\numberwithin{equation}{section}

\title{Global $W^{2,1+\eps}$ estimates for Monge-Amp\`ere equation with natural boundary condition} 

\author{Ovidiu Savin}
\address{Department of Mathematics,	Columbia University, New York, USA}
\email{savin@math.columbia.edu}

\author{Hui Yu}
\address{Department of Mathematics,	Columbia University, New York, USA}
\email{ huiyu@math.columbia.edu}
\thanks{O.~S.~is supported by  NSF grant DMS-1500438.}

\begin{document}

\begin{abstract}
For the \MA equation with a right-hand side bounded away from $0$ and infinity, we show that the solution,  subject to the natural boundary condition arising in optimal transport,  is in $W^{2,1+\eps}$ up to the boundary. 
\end{abstract}

\maketitle

%\tableofcontents
%%%%%%%%%%%%%%%%%%%%%%%%%%%%%%%%%%%%%%%%%%%%%%%%%%%%%%%%%%%%%%%%%%%%%%%%%%%%%%%%%%%%%%%%%%%%%%%%%%%%%%
\section{Introduction}
Let $\Omega$ and $\Omega^*$ be two bounded convex domains in $\R^d$, and $f$ be a function on $\Omega$ satisfying \begin{equation}\label{fBound}\frac{1}{\Lambda}\le f\le\Lambda\end{equation}for some positive constant $\Lambda$ . In this work, we study the regularity of convex Alexandrov solutions to the following problem 
\begin{equation}\label{BVP}
\begin{cases}\det(D^2u)=f &\text{ in $\Omega$,}\\
\nabla u(\Omega)=\Omega^*.&
\end{cases}
\end{equation}
For the definition of Alexandrov solutions to the \MA equation, the reader can consult Figalli \cite{F1} or Giut\'errez \cite{G}. Here we point out that this is the natural boundary value problem arising from the study of the theory of optimal transport.  

To be precise, suppose $\nu$ and $\nu^*$ are two probability measures supported on $\Omega$ and $\Omega^*$ with density functions $g$ and $g^*$ respectively, then \eqref{BVP} is satisfied by the potential of the optimal transport that pushes-forward $\nu=gdx$ to $\nu^*=g^*dx$ \cite{B,V}. In this case, the right-hand side is $f=\frac{g}{g^*\circ\nabla u}.$

When $f$ is continuous, the regularity of solutions to \eqref{BVP} has been studied extensively. Caffarelli showed that $u$ is locally in $W^{2,p}$ in the interior of $\Omega$ for all $p>0$ \cite{C1}.  If $f$ is further assumed to be H\"older continuous, Caffarelli showed that $D^2u$ is H\"older continuous in the interior of $\Omega$.
When the domains are $C^{1,1}$, Chen-Liu-Wang \cite{CLW} proved that these estimates hold up to the boundary of $\Omega$, based on earlier results by Caffarelli \cite{C3} and Urbas \cite{U}. In two dimensions, we recently established the optimal global $W^{2,p}$ estimate without any regularity assumptions on the domains except their convexity \cite{SY}. Still in two dimension, if the domains are assumed to be $C^{1,\alpha}$, $D^2u$ is shown to be H\"older continuous by Chen-Liu-Wang \cite{CLW2}.

For several important applications, however, it is necessary to understand the regularity of $u$ when $f$ fails to be continuous. In the optimal transport problem described above, $f$ does not enjoy any regularity if the density functions $g$ and $g^*$ are only assumed to be bounded away from $0$ and infinity. This problem also has deep implications in the study of semi-geomstrophic equations \cite{F2}.

When $f$ satisfies \eqref{fBound} but is allowed to be  discontinuous, much less is known about the regularity of $u$. Caffarelli showed that $u$ is $C^{1,\alpha_0}$ up to the boundary for some small dimensional $\alpha_0$ \cite{C2}. In terms of Sobolev regularity, Wang \cite{W} showed that for any $p>1$, one can find sufficiently large $\Lambda$ such that $u$ fails to be in $W^{2,p}$ even in the interior of the domain.  Nevertheless, for fixed $\Lambda$, De Philippis-Figalli \cite{DF} was able to show that $u$ is in $W^{2,1}$ in the interior of $\Omega$. This was later improved to an interior $W^{2,1+\eps}$-estimate  independently by De Philippis-Figalli-Savin \cite{DFS} and Schmidt \cite{Sch}.

In this work, we extend this interior $W^{2,1+\eps}$-estimate up to the boundary.  To be precise, our main result is

\begin{thm}\label{Result}
Suppose $\Omega$ and $\Omega^*$ are bounded convex domains in $\R^d.$ Let $u$ be an Alexandrov solution to \eqref{BVP} with $\frac{1}{\Lambda}\le f\le\Lambda$ for some positive constant $\Lambda.$

Then there are positive constants $\eps$, depending only on $d$ and $\Lambda$,  and $C$, further depending on the inner and outer radii of $\Omega$ and $\Omega^*$,  such that \begin{equation*}
\int_\Omega |D^2u|^{1+\eps}\le C. 
\end{equation*} 
\end{thm} 

The exponent $1+\eps$ is optimal due to the examples of Wang \cite{W}. Also, the result is sharp in the sense that the estimate has to depend on $d, \Lambda$  and the inner and outer radii of $\Omega$ and $\Omega^*.$ 

We'd like to point out  that no regularity of $\Omega$ and $\Omega^*$ is assumed. In this case, it remains an interesting problem whether a global $W^{2,p}$-estimate can be established in the spirit of \cite{SY}.

This paper is structured as follows: in Section 2, we introduce some notations and collect some useful preliminary results. In Section 3, we give estimates in the normalized picture. The scaled versions of these estimates are applied to our solution $u$ in Section 4. In the last section we give the proof Theorem \ref{Result}.

%%%%%%%%%%%%%%%%%%%%%%%%%%%%%%%%%%%%%%%%%%%%%%%%%%%%%%%%%%%%%%%%%%%%%%%%%%%%%%%%%%%%%%%%%%%%%%%%%%%%%%
\section{Preliminaries}
\subsection{Extension of the solution to $\R^d$.}
Let $u$ be an Alexandrov solution to \eqref{BVP}, we can extend it to the entire $\R^d$ by $$x\in\R^d\mapsto\sup_{y\in\Omega}(u(y)+\nabla u(y)\cdot (x-y)).$$ The resulting function, still denoted by $u$, is a convex function solving the following equation  in the Alexandrov sense \cite{C2}:
\begin{equation}\label{GlobalBVP}
\begin{cases}\det(D^2u)=f\chi_\Omega &\text{ in $\R^d$,}\\
\nabla u(\Omega)=\Omega^*.&
\end{cases}
\end{equation} 
For a set $S$, $\chi_S$ denotes its characteristic function. 

We  assume $u\in C^2(\overline{\Omega})$ in the rest of the paper, and prove Theorem \ref{Result} for such solutions. This implies the estimate for general solutions via a standard approximation procedure. 

\subsection{Sections and their properties.}
Sections are a fundamental tool in the study the \MA equation. Among several related notions of sections, the centered section introduced in \cite{C1}  is the most convenient for our purpose. We give its definition here.

\begin{defi}
Let $x_0\in\overline{\Omega}$, $h>0$,  \textit{the centered section} of $u$ of height $h$ at $x_0$ is defined by
$$S_h[u](x_0)=\{y\in\R^d|y<u(x_0)+p\cdot(y-x_0)+h\}.$$ Here $p\in\R^d$ is chosen such that the center of mass of $\Shu(x_0)$ stays at $x_0$, that is, $$\frac{1}{|\Shu(x_0)|}\int_{\Shu(x_0)}ydy=x_0.$$ For the existence of such $p$, see \cite{C2}.
\end{defi} 

By the convexity of $u$, these sections are bounded convex subsets of $\R^d$. In order to describe their shapes, we need the following lemma due to Fritz John \cite{J}:

\begin{lem}[John's lemma]\label{John'sLemma}For  any  bounded convex subset $S$ of $\R^d$, there is an ellipsoid $E$ with the same center of mass as $S$ such that $$E\subset S\subset \alpha_d E.$$This factor $\alpha_d$ depends only on the dimension $d$.  

\end{lem} 

Such ellipsoid $E$ is called the John ellipsoid of $S$.

For a set $S$ and a positive constant $c$, $cS$ denotes the dilation of $S$ by a factor of $c$ with respect to the center of mass of $S$. 

When $E$ is an ellipsoid, we write $E=x+\sum\lambda_j\omega_j$ when $x$ is the center of $E$, $\omega_j$'s are the directions of the principal axises of $E$, and $\lambda_j$'s are the length of the axis in the direction of $\omega_j.$

To each such ellipsoid  $E=x+\sum\lambda_j\omega_j$, we associate the matrix \begin{equation}\label{ME}M_E=\sum\frac{1}{\lambda_j}\omega_j\otimes\omega_j,
\end{equation}where $\otimes$ denotes the tensor product. This is the matrix that maps $E$ to a unit ball.

Sections share many properties with Euclidean balls. In particular, one has a Besicovitch-type covering lemma with sections. The following is based on Caffarelli-Giuti\'errez \cite{CG}:

\begin{lem}\label{CoveringLemma}Let $A$ be a subset of $\Omega$. Suppose for each $x\in A$, a section $S_{h_x}[u](x)$ is chosen such that the heights $h_x$ are uniformly  bounded. Let $\mathcal{F}$ denote this family of sections. 

There are constants $\eta_0\in(0,1)$ and $K$, depending only on $d$ and  $\Lambda$, such that there is a countable subfamily $\{S_{h_j}[u](x_j)\}$ of $\mathcal{F}$ satisfying the following: 
\begin{enumerate}
\item{$A\subset \cup S_{h_j}[u](x_j)$;}
\item{$\sum\chi_{S_{\eta_0h_j}[u](x_j)}\le K$. }
\end{enumerate}
\end{lem}

%%%%%%%%%%%%%%%%%%%%%%%%%%%%%%%%%%%%%%%%%%%%%%%%%%%%%%%%%%%%%%%%%%%%%%%%%%%%%%%%%%%%%%%%%%%%%%%%%%%%%%%%%%%%%%%%%%%%%%%%%%%%%%%%%%%%%%%%%%%%%%%%%%%%%%%%%%%%
\section{Estimates for normalized solutions}
In this section we establish several key estimates in the normalized picture. Later these are applied to our solution after rescaling. The methods are motivated by De Philippis-Figalli \cite{DF} and De Philippis-Figali-Savin \cite{DFS}. However, since we are dealing with global regularity estimates, we need more detailed analysis concerning the interaction between the sections and the boundary of the domain. 

The following {\bf assumptions} are in effect throughout this section:\begin{enumerate}
\item{$U$ is a convex domain in $\R^d$ containing a point $x_0$;}
\item{$v$ is a $C^2$ convex function in $\overline{U}$ extended to $\R^d$ by $$x\mapsto \sup_{y\in U}(v(y)+\nabla v(y)\cdot(x-y));$$}
\item{$Z:=\{v<0\}$ is centered at $x_0$ and normalized in the sense that $B_1(x_0)\subset Z\subset B_{\alpha_d}(x_0)$,  where $\alpha_d$ is the constant  in Lemma \ref{John'sLemma};}
\item{$|\nabla v|\le L_0$ in $Z$;}
\item{$\det(D^2v)=g\chi_U$ in $Z$ with $\frac{1/\Lambda}{|Z\cap U|}\le g\le\frac{\Lambda}{|Z\cap U|}$;}
\item{$E=\sum\lambda_j e_j$ is the John ellipsoid for $Z\cap U.$ Here $\{e_j\}$ is the standard basis of $\R^d$. In particular, $Z\cap U$ is centered at $0$. }
\end{enumerate}

Throughout this section, constants depending only on $d$, $\Lambda$ and $L_0$ are called \textit{universal constants.} 

Denote $$h_0=|\inf_Z v|,$$ assumptions (3) and (5) imply that $0<c\le h_0\le C$ for some universal $c$ and $C$. 

Inside $Z\cap U$, we expect $v$ to behave like the parabola $p(x)=(M_Ex)\cdot x$, where $M_E$ is the matrix defined in \eqref{ME}. An application of the ABP estimate \cite{G} shows that this is indeed true in a large portion of $Z$:

\begin{lem}\label{ABP}Let $Z_{\eta_0}=S_{\eta_0h_0}[v](x_0)$, where $\eta_0$ is the constant in Lemma \ref{CoveringLemma}. Then there are universal constants $C_0$ and $\delta_0$ such that \begin{equation}\label{ABPEstimate}
\frac{|Z_{\eta_0}\cap U\cap \{C_0^{-1}M_E\le D^2v\le C_0M_E\}|}{|Z\cap U|}\ge\delta_0.
\end{equation} 
\end{lem} 

\begin{proof}\textit{Step 1: Construction of comparison functions.}
By the engulfing property \cite{G}, there is a universal constant $0<c_0<1$ such that $c_0Z\subset Z_{\eta_0}\subset Z.$ Consequently,  the ellipsoid $\tilde{E}=c_0\alpha_d\sum\lambda_je_j$ satisfies $\frac{1}{\alpha_d}\tilde{E}\subset Z_{\eta_0}\cap U\subset\tilde{E}.$

Define a quadratic polynomial $\tilde{p}:\tilde{E}\to\R$ by $\tilde{p}(x)=\sum\frac{1}{c_0\alpha_d\lambda_j}x_j^2$, and extend $\tilde{p}$ to the entire $\R^d$ by $\tilde p(x)=\sup_{y\in\tilde{E}}(\tilde{p}(y)+\nabla \tilde{p}(y)\cdot(x-y)).$ 

Then one has $$\sup_{y\in\R^d}|\nabla \tilde{p}(y)|\le \sup_{y\in\tilde{E}}|\nabla \tilde{p}(y)|\le 1,$$ and \begin{equation}\label{BoundForP}0\le \tilde{p}\le C \text{ in $Z_{\eta_0}$}\end{equation}for some universal $C$.  

Up to subtracting an affine function, we have $v=\eta_0h_0$ along $\partial Z_{\eta_0}$, and $v(x_0)=0.$

In particular, if we define $p=\frac{1}{2C}\eta_0h_0\tilde{p}+\frac{1}{2}\eta_0h_0$, where $C$ is the constant in \eqref{BoundForP},  then $$p\le v \text{ on $\partial Z_{\eta_0}$}$$ and $$\frac{1}{2}\eta_0h_0\le p\le \eta_0h_0 \text{ in $Z_{\eta_0}$.}$$

Let $w=v-p$.  

Then $$w\ge 0 \text{ on $\partial Z_{\eta_0},$}$$ and $$|\inf_{Z_{\eta_0}}w|\ge \frac{1}{2}\eta_0h_0.$$

\textit{Step 2: The ABP estimate.}
If $\Gamma_w$ is the convex envelop of $w$ in $Z_{\eta_0}$, then the ABP estimate \cite{G} implies 
\begin{equation}\label{1}ch_0^d\le|\nabla\Gamma_w(Z_{\eta_0}\cap \{\Gamma_w=w\})|\end{equation} for some dimensional $c$. 

For $\bar{x}\in Z_{\eta_0}\cap \{\Gamma_w=w\}$,  there is an affine function $\ell$ such that $$\ell+p\le v \text{ in $Z_{\eta_0}$}$$ and $$\ell(\bar{x})+p(\bar{x})=v(\bar{x}).$$ In particular, one has $$\nabla\ell(\bar{x})+\nabla p(\bar{x})=\nabla v(\bar{x}).$$

\textit{Step 3: Localizing  to $U$.} By assumption (2) at the beginning of this section, either $\bar{x}\in U$, or there is a point $\bar{y}\in Z_{\eta_0}\cap U$ such that $\nabla v(\bar{y})=\nabla v(\bar{x})$, and that $v$ is affine along the line segment between $\bar{x}$ and $\bar{y}$. 

By convexity, one has \begin{align*}\ell(\bar{y})+p(\bar{y})&\ge (\ell+p)(\bar{x})+\nabla(\ell+p)(\bar{x})\cdot (\bar{y}-\bar{x})\\&=v(\bar{x})+\nabla v(\bar{x})\cdot(\bar{y}-\bar{x})\\&=v(\bar{y}).\end{align*}Together with $\ell+p\le v \text{ in $Z_{\eta_0}$}$, this implies $\bar{y}\in Z_{\eta_0}\cap\{w=\Gamma_w\}$. 

In particular, $$\Gamma_w(\bar{x})=\nabla\ell(\bar{x})=\nabla\ell(\bar{y})\in\Gamma_w(Z_{\eta_0}\cap U\cap\{w=\Gamma_w\} ).$$ Since this is true for all $\bar{x}\in Z_{\eta_0}\cap \{\Gamma_w=w\}$, we conclude $$\nabla\Gamma_w(Z_{\eta_0}\cap \{\Gamma_w=w\})\subset \nabla\Gamma_w(Z_{\eta_0}\cap U\cap \{\Gamma_w=w\}).$$

\textit{Step 4: Proof of \eqref{ABPEstimate}.} Using this inclusion in \eqref{1} and note that $D^2\Gamma_w\le D^2 v$ in $\{\Gamma_w=w\}$, we have 
\begin{align*}
ch_0^d&\le |\nabla\Gamma_w(Z_{\eta_0}\cap U\cap \{\Gamma_w=w\})|\\&\le \int_{Z_{\eta_0}\cap U\cap \{\Gamma_w=w\}}\det(D^2\Gamma_w)\\&\le \int_{Z_{\eta_0}\cap U\cap \{\Gamma_w=w\}}\det(D^2 v)\\&\le\frac{\Lambda}{|Z\cap U|}|Z_{\eta_0}\cap U\cap \{\Gamma_w=w\}|.
\end{align*}For the last inequality, we used assumption (5) at the beginning of this section. 

Since $h_0$ is universal, the estimate above implies $$|Z_{\eta_0}\cap U\cap \{\Gamma_w=w\}|\ge\delta_0|Z\cap U|$$ for some universal $\delta_0>0.$

To get \eqref{ABPEstimate},  it suffices to note that in $Z_{\eta_0}\cap U\cap \{\Gamma_w=w\}$, one has $$D^2v\ge D^2p=cM_E$$ for some universal $c.$
\end{proof}

We now use the previous lemma to estimate the integral of pure second derivatives of $v$ in $Z\cap U$ in terms of the integral over `the good set'. We first estimate second order derivatives in the directions along the axises of $\R^d$:

\begin{lem}\label{vjj}
Under the same assumptions as in Lemma \ref{ABP}, there is a universal constant $C$ such that \begin{equation}\label{2}
\int_{Z\cap U}v_{jj}\le C\int_{Z_{\eta_0}\cap U\cap \{C_0^{-1}M_E\le D^2v\le C_0M_E\}}v_{jj}
\end{equation}  for each $j=1,2,\dots, d$.
\end{lem} 
\begin{proof}
To simply our notations, let's denote `the good set' by $$G=Z_{\eta_0}\cap U\cap \{C_0^{-1}M_E\le D^2v\le C_0M_E\}.$$

With Lemma \ref{ABP}, the right-hand side of \eqref{2} can be  bounded from below by 
\begin{align*}
\int_G v_{jj}&\ge C_0^{-1}\frac{1}{\lambda_j}|G|\\&\ge C_0^{-1}\frac{1}{\lambda_j}\delta_0|Z\cap U|.\end{align*}Since $E=\sum\lambda_je_j$ is the John ellipsoid for $Z\cap U$, we have $$|Z\cap U|\ge c\lambda_1\lambda_2\dots\lambda_d$$ for some dimensional $c$.
As a result,  \begin{equation}\label{3}
\int_G v_{jj}\ge c\lambda_1\lambda_2\dots\lambda_d/\lambda_j
\end{equation} for some universal $c$. 

Now we estimate the left-hand side of \eqref{2}.

Define $v_{M_E}(x)=v(M_E^{-1}x)=v(\lambda_1x_1,\lambda_2x_2,\dots,\lambda_dx_d)$, then $$\frac{\partial}{\partial x_j}v_{M_E}(x)=\lambda_j\frac{\partial}{\partial x_j}v(M^{-1}x).$$ By assumption (4) at the beginning of this section, we have  $|\frac{\partial}{\partial x_j}v_{M_E}|\le L_0\lambda_j$ in $M_E(Z)$.

The left-hand side of \eqref{2} can be computed as  \begin{align*}
\int_{Z\cap U}v_{jj}&=\int_{M_E(Z\cap U)}v_{jj}(M_E^{-1}x)\det(M_E^{-1})dx\\&=\det(M_E^{-1})\int_{M_E(Z\cap U)}\frac{1}{\lambda_j^2}\frac{\partial^2}{\partial x_j^2}v_{M_E}(x)dx\\&=\det(M_E^{-1})\frac{1}{\lambda_j^2}\int_{\partial M_E(Z\cap U)}\frac{\partial}{\partial x_j}v_{M_E}\nu\cdot e_j,
\end{align*}where $\nu$ is the outward unit normal to $M_E(Z\cap U)$. 

Consequently, \begin{align*}\int_{Z\cap U}v_{jj}&\le \det(M_E^{-1})\frac{1}{\lambda_j^2}C\lambda_j\mathcal{H}^{d-1}(\partial M_E(Z\cap U))\\&\le C\lambda_1\lambda_2\dots\lambda_d/\lambda_j\end{align*} for some universal $C$. Here we used $B_1\subset M_E(Z\cap U)\subset B_{\alpha_d}$ to control  the $(d-1)$-dimensional Hausdorff measure of $\partial M_E(Z\cap U)$.

Combining this with \eqref{3}, we get the desired estimate. \end{proof} 

For a general vector $\xi=\sum \xi_je_j$ in $\R^d$, define $v_{\xi\xi}:=(D^2v\xi)\cdot\xi.$ A similar estimate as the one in Lemma \ref{vjj} holds for these second order derivatives.  
\begin{lem}\label{vee}
Under the same assumptions as in Lemma \ref{ABP}, there is a universal constant $C$ such that 
\begin{equation}\label{4}
\int_{Z\cap U}v_{\xi\xi}\le C\int_{Z_{\eta_0}\cap U\cap \{C_0^{-1}M_E\le D^2v\le C_0M_E\}}v_{\xi\xi}.
\end{equation} 
\end{lem} 

\begin{proof}Again we write $$G=Z_{\eta_0}\cap U\cap \{C_0^{-1}M_E\le D^2v\le C_0M_E\}.$$

By convexity of $v$, $D^2v\ge 0$.  Thus for $1\le i, j\le d$ one has $$v_{ii}v_{jj}\ge v_{ij}^2.$$

Therefore, \begin{align*}v_{\xi\xi}&=\sum_{j=1}^dv_{jj}\xi_j^2+\sum_{i\neq j}v_{ij}\xi_i\xi_j\\&\le \sum_{j=1}^dv_{jj}\xi_j^2+\frac{1}{2}\sum_{i\neq j}(v_{ii}\xi_i^2+v_{jj}\xi_j^2)\\&\le d\sum v_{jj}\xi_j^2.\end{align*}

Combining this with Lemma \ref{vjj}, we can estimate the left-hand side of \eqref{4} as 
\begin{align*}
\int_{Z\cap U}v_{\xi\xi}&\le d\sum\xi_j^2\int_{Z\cap U}v_{jj}\\&\le Cd\sum\xi_j^2\int_Gv_{jj}.
\end{align*}
Now note that on $G$, $C_0^{-1}M\le D^2v\le C_0M$. Thus $$ \sum v_{jj}\xi_j^2\le C_0\sum\frac{1}{\lambda_j}\xi_j^2$$ and $$(D^2v\xi)\cdot\xi\ge C_0^{-1}\sum\frac{1}{\lambda_j}\xi_j^2.$$

Therefore we can continue the previous estimate by  
\begin{align*}
\int_{Z\cap U}v_{\xi\xi}&\le Cd\sum\xi_j^2\int_Gv_{jj}\\&\le CdC_0^{2}\int_G v_{\xi\xi}.
\end{align*}
This is the desired estimate. 
\end{proof}

%%%%%%%%%%%%%%%%%%%%%%%%%%%%%%%%%%%%%%%%%%%%%%%%%%%%%%%%%%%%%%%%%%%%%%%%%%%%%%%%%%%%%%%%%%%%%%%%%%%%%%%%%%%%%%%%%%%%%%%%%%%%%%%%%%%%%%%%%%%%%%%%%%%%%%%%%%%%
\section{Estimates in sections of our solution}
In this section we rescale the estimates from the previous one, so that they can be applied to our solution $u$. These computations are more or less standard. Nevertheless, we include them here for completeness.

\begin{lem}\label{uee}
Suppose $u$ is a solution to \eqref{GlobalBVP}.  For a point $x_0\in\Omega$ and $h>0$, let $A$ be the symmetric matrix such that $B_1(x_0)\subset A(\Shu(x_0))\subset B_{\alpha_d}(x_0).$

Up to rotation and translation, suppose $E=\sum\lambda_je_j$ is the John ellipsoid for $\Shu(x_0)\cap\Omega$. 

Then there are constants $C$ and $C_0$, depending only on $d$ and $\Lambda$, such that for all $\xi\in\R^d$

\begin{equation*}
\int_{\Shu(x_0)\cap\Omega}u_{\xi\xi}\le C\int_Gu_{\xi\xi},
\end{equation*} where $G={S_{\eta_0h[u](x_0)}\cap\Omega\cap\{C_0^{-1}hAM_E\le D^2u\le C_0hAM_E\}}.$\end{lem} 

The existence of this normalizing $A$ is a consequence of Lemma \ref{John'sLemma}. 

Here $\eta_0$ is the constant in Lemma \ref{CoveringLemma}. $M_E$ is the matrix defined in \eqref{ME}.

\begin{proof}
Let $p$ be the vector such that $\Shu(x_0)=\{y|u(y)<u(x_0)+p\cdot(y-x_0)+h\}$. Let $\ell$ be the affine function $x\mapsto p\cdot(x-x_0)+h$. 

Define $v(x)=\frac{1}{h}(u-\ell)(A^{-1}x)$, $Z=A(\Shu(x_0))$, $U=A(\Omega)$ and $g(x)=\frac{1}{|Z\cap U|}f(A^{-1}x)$. Then it is not difficult to see the assumptions (1)-(3) and (5) at the beginning of Section 3 are satisfied, up to a dimensional change of the value of $\Lambda$.  

Moreover, by the doubling property, $Z'=\{v<1\}$ is at a positive distance to $Z$, where the distance depending only on $d$ and $\Lambda$. Therefore, the value $L_0$ as in assumption (4) in Section 3 depends only on $d$ and $\Lambda.$

Let $\tilde{E}=A(E)$. Then $\tilde{E}$ is the John ellipsoid for $Z\cap U$ as in assumption (6) at the beginning of Section 3. Suppose $\tilde{E}=\sum\tilde{\lambda_j}\tilde{e_j}$. Denote by $\tilde{M}$ the matrix $\sum\frac{1}{\tilde{\lambda_j}}\tilde{e_j}\otimes\tilde{e_j}$, where $\otimes$ denotes the tensor product. Then Lemma \ref{vee}, applied to the direction $\tilde{\xi}=A\xi$, gives
$$\int_{A(\Shu(x_0)\cap\Omega)}v_{\tilde{\xi}\tilde{\xi}}\le C\int_{A(S_{\eta_0h}[u](x_0)\cap\Omega)\cap \{C_0^{-1}\tilde{M}\le D^2v\le C_0\tilde{M}\}}v_{\tilde{\xi}\tilde{\xi}}$$
 for some $C$ and $C_0$ depending only on $d$ and $\Lambda$.
 
 Back to the original variables, this means
 $$\int_{\Shu(x_0)\cap\Omega}(D^2u \xi)\cdot \xi\le C\int_{\tilde{G}}(D^2u \xi)\cdot\xi,$$where $$\tilde{G}=S_{\eta_0h}[u](x_0)\cap\Omega\cap\{C_0^{-1}hA\tilde{M}A\le D^2u\le C_0hA\tilde{M}A\}.$$
 
To conclude,  it suffices to note that $\tilde{M}A=M_E.$\end{proof} 

Up to a dimensional constant, this can be upgraded to an estimate for the integral of $|D^2u|$:
\begin{prop}\label{uHessian}
Under the same assumptions as in Lemma \ref{uee}, there are constants $C$ and $C_0$, depending only on $d$ and $\Lambda$, such that 
$$\int_{\Shu(x_0)\cap\Omega}|D^2u|\le C\int_G|D^2u|$$ for $G=S_{\eta_0h[u](x_0)}\cap\Omega\cap\{C_0^{-1}hAM_E\le D^2u\le C_0hAM_E\}.$
\end{prop} 

\begin{proof}
By summing up the estimate in Lemma \ref{uee} in $d$ orthogonal directions, we get a similar estimate where the integrand is $\Delta u$.  From here it suffices to note that for convex functions $|D^2u|\le \Delta u\le d|D^2u|.$
\end{proof} 

Under the assumptions as in Lemma \ref{uee}, the matrix that defines the `good set' $G$ associated with $\Shu(x_0)$ is $hAM_E.$ The next result says that this matrix has the correct behaviour when $h$ is large and when $h\to 0$. 

To simplify our notations, let's define the matrix $A_h$ and $M_h$ to be the matrices $A$ and $M_E$ as in Lemma \ref{uee} for the section $\Shu(x_0).$

 Let $T_h=hA_hM_h$.
Then one has 
\begin{prop}\label{CorrectBehaviour}
There is a constant $C_1$, depending only on $d$, $\Lambda$, and the inner and outer radii of $\Omega$ and $\Omega^*$, such that $$\frac{1}{C_1}\le T_1\le C_1.$$

There is a constant $C_2$, depending only on $d$, such that for $h>0$ small, $$\frac{1}{C_2}D^2u(x_0)\le T_h\le C_2D^2u(x_0).$$ 
\end{prop} 

\begin{proof}
By Lipschitz estimate and uniform strict convexity of $u$ \cite{C2}, we have $$B_{r_1}(x_0)\subset S_1[u](x_0)\subset B_{R_1}(x_0)$$ for some $r_1$ and $R_1$ depending on $d$, $\Lambda$, and inner and outer radii of $\Omega$ and $\Omega^*$.

Consequently, $A_1$ is  bounded from both sides by constants depending only on $d$, $\Lambda$, and inner and outer radii of $\Omega$ and $\Omega^*$.

Meanwhile, $S_1[u](x_0)\cap\Omega\subset B_{R_1}(x_0)$ and $|S_1[u](x_0)\cap\Omega|\ge |B_{r_1}(x_0)\cap\Omega|\ge c|B_{r_1}(x_0)|$ for some $c$ depending only on the inner and outer radii of $\Omega$.

As a result,  the John ellipsoid for $S_1[u](x_0)\cap\Omega$ has a diameter that is  bounded from above and a volume that is  bounded from below. Thus $M_h$ is also  bounded from both sides by constants depending only on $d$, $\Lambda$, and inner and outer radii of $\Omega$ and $\Omega^*$.

Therefore, $\frac{1}{C_1}\le T_1=A_1M_1\le C_1$ for some $C_1$ depending only on $d$, $\Lambda$, and inner and outer radii of $\Omega$ and $\Omega^*$.

To see the second statement in the proposition, we first note that when $h$ is small, $\Shu(x_0)\subset\Omega$ since $x_0\in\Omega$. 

Consequently $A_h=M_h$ for small $h$ and $T_h=hA_h^2.$ 

By $C^2$ regularity of $u$ inside $\Omega$, up to subtracting an affine function, $$u(x)=(D^2u(x_0)(x-x_0))\cdot(x-x_0)+o(|x-x_0|^2).$$ Up to a rotation, $$D^2u(x_0)=\begin{pmatrix}\eta_1&0&0&\dots\\0&\eta_2&0&\dots\\0&0&\eta_3&\dots\\\vdots&\ddots\\0 &0&\dots&\eta_d\end{pmatrix}.$$
Then $\Shu(x_0)$ is comparable to $x_0+\sum(\frac{h}{\eta_j})^{1/2}e_j.$ Therefore, up to a dimensional constant, $A_h$ is comparable to $$\begin{pmatrix}(\frac{\eta_1}{h})^{1/2}&0&0&\dots\\0&(\frac{\eta_2}{h})^{1/2}&0&\dots\\0&0&(\frac{\eta_3}{h})^{1/2}&\dots\\\vdots&\ddots\\0 &0&\dots&(\frac{\eta_d}{h})^{1/2}\end{pmatrix}.$$
 
As a result, $T_h=hA_h^2$ is comparable $D^2u(x_0)$ up to a dimensional constant for small $h$. \end{proof}

%%%%%%%%%%%%%%%%%%%%%%%%%%%%%%%%%%%%%%%%%%%%%%%%%%%%%%%%%%%%%%%%%%%%%%%%%%%%%%%%%%%%%%%%%%%%%%%%%%%%%%%%
\section{Proof of Theorem \ref{Result}}
In this final section of the paper, we give the proof of the main result. 

To simplify our notations, define $\kappa=\max\{C_0C_2,C_0^2\}$, where $C_0$ is the constant in Proposition \ref{uHessian} and $C_2$ is the constant in Proposition \ref{CorrectBehaviour}. In particular, $\kappa$ depends only on $d$ and $\Lambda.$

For each integer $m$, let's define $$D_m=\{x\in\Omega||D^2u(x)|\ge \kappa^m\}.$$ The $W^{2,1+\eps}$-estimate is a direct consequence of the following lemma concerning the decay of integrals over $D_m$:

\begin{lem}\label{LevelSetDecay}
Suppose $u$ is a solution to \eqref{GlobalBVP}. There is a constant $\tau\in(0,1)$, depending only on $d$ and $\Lambda$, such that    $$\int_{D_{m+1}}|D^2u|\le (1-\tau)\int_{D_m}|D^2u|$$for each $m\ge m_0$. 

Here $m_0$ is an integer depending on $d$, $\Lambda$, and the inner and outer radii of $\Omega$ and $\Omega^*$.
\end{lem} 

\begin{proof}\textit{Step 1: Covering $D_{m+1}$ by sections with the correct height.}
For $x\in D_{m+1}$, $|D^2u(x)|\ge \kappa^{m+1}$. By Proposition \ref{CorrectBehaviour}, $|T_h|$ ranges from $1/C_1$ to $\kappa^{m+1}/C_2$ as $h$ changes from $1$ to $0$.  $T_h$ is the matrix defined before Proposition \ref{CorrectBehaviour}.

By our choice of $\kappa$, $C_0\kappa^{m}\le \kappa^{m+1}/C_2$. We can also choose $m_0$, depending also on the inner and outer radii of $\Omega$ and $\Omega^*$, such that $C_0^{-1}\kappa^{m+1}\ge C_0^{-1}\kappa^{m_0+1}\ge 1/C_1.$ 

Thus we can pick $h_x>0$ such that $$C_0\kappa^m\le |T_{h_x}|< C_0^{-1}\kappa^{m+1}.$$ 

Let $\mathcal{F}$ denote the family of sections corresponding to such choice of heights, namely, $\mathcal{F}=\{S_{h_x}[u](x)\}_{x\in D_{m+1}}.$ Then Lemma \ref{CoveringLemma} gives a countable subfamily $\{S_{h_j}[u](x_j)\}$ such that 
$$D_{m+1}\subset \cup S_{h_j}[u](x_j), \text{ and} \sum\chi_{S_{\eta_0h_j}[u](x_j)}\le K.$$

\textit{Step 2: Estimate in each section.} Let $S_{h}[u](x)$ denote a generic section in this subfamily.  Then Proposition \ref{uHessian} implies $$\int_{\Shu(x)\cap\Omega}|D^2u|\le C\int_G|D^2u|$$ where $G=S_{\eta_0h[u](x)}\cap\Omega\cap\{C_0^{-1}T_h\le D^2u\le C_0T_h\}.$ 

Now by our choice of $h$, $C_0\kappa^m\le |T_{h}|< C_0^{-1}\kappa^{m+1}.$ In particular, we have 
$$G\subset S_{\eta_0h[u](x)}\cap\Omega\cap\{\kappa^m\le |D^2u|< \kappa^{m+1}\}\subset S_{\eta_0h[u](x)}\cap (D_m\backslash D_{m+1}).$$

Hence the previous estimate leads to \begin{equation*}
\int_{\Shu(x)\cap\Omega}|D^2u|\le C\int_{S_{\eta_0h[u](x)}\cap (D_m\backslash D_{m+1})}|D^2u|.
\end{equation*} 

\textit{Step 3: The covering argument.} With this estimate and the two properties at the end of Step 1, we have the following 
\begin{align*}
\int_{D_{m+1}}|D^2u|&\le\sum\int_{S_{h_j}[u](x_j)\cap\Omega}|D^2u|\\&\le \sum C\int_{S_{\eta_0h_j[u](x_j)}\cap (D_m\backslash D_{m+1})}|D^2u|\\&=C\int_{D_m\backslash D_{m+1}}|D^2u|\sum\chi_{S_{\eta_0h_j}[u](x_j)}\\&\le CK\int_{D_m\backslash D_{m+1}}|D^2u|\\&=CK\int_{D_m}|D^2u|-CK\int_{D_{m+1}}|D^2u|.
\end{align*}
Consequently, $$\int_{D_{m+1}}|D^2u|\le \frac{CK}{1+CK}\int_{D_{m}}|D^2u|,$$where $C$ and $K$ are constants depending only on $d$ and $\Lambda.$
\end{proof}

Now Theorem \ref{Result} follows from a standard iteration:
\begin{proof}[Proof of Theorem \ref{Result}]
For $t>\kappa^{m_0}$, we find an integer $k$ such that $$\kappa^{m_0+k}\le t< \kappa^{m_0+k+1},$$ that is, $k\le \log_\kappa t-m_0<k+1.$

An iteration of Lemma \ref{LevelSetDecay} gives 
\begin{align*}\int_{\{|D^2u|\ge t\}\cap\Omega}|D^2u|&\le  \int_{D_{m_0+k}}|D^2u|\\&\le (1-\tau)^{k}\int_{D_{m_0}}|D^2u|\\&\le (1-\tau)^{-1-m_0}\cdot t^{\log_\kappa(1-\tau)}\cdot\int_{D_{m_0}}|D^2u|.
\end{align*}

Note that $$\int_{D_{m_0}}|D^2u|\le\int_{\Omega}\Delta u=\int_{\partial\Omega}\nabla u\cdot \nu\le C$$ for come $C$ depending only on $d$ and the inner and outer radii of $\Omega$ and $\Omega^*,$ the previous estimate gives
$$\int_{\{|D^2u|\ge t\}\cap\Omega}|D^2u|\le Ct^{-\eps_0},$$ where $\eps_0>0$ depends only on $d$ and $\Lambda$, and $C$ depends further on the inner and outer radii of $\Omega$ and $\Omega^*.$

By Markov's inequality, this gives $|\{|D^2u|\ge t\}\cap\Omega|\le Ct^{-1-\eps_0}$ for $t\ge \kappa^{m_0}$.

Therefore, we can pick $\eps\in(0,\eps_0)$ depending only on $d$ and $\Lambda$. Then it follows \begin{align*}
\int_\Omega|D^2u|^{1+\eps}&=\int_{D_{m_0}}|D^2u|^{1+\eps}+\int_{\{|D^2u|\le \kappa^{m_0}\}\cap\Omega}|D^2u|^{1+\eps}\\&\le C_\eps\int_{\kappa^{m_0}}^\infty t^{\eps}|\{|D^2u|\ge t\}\cap\Omega|dt+\kappa^{m_0(1+\eps)}|\Omega|\\&\le C\int_{\kappa^{m_0}}^\infty t^{-1-\eps_0+\eps}dt+\kappa^{m_0(1+\eps)}|\Omega|\\&=C\kappa^{m_0(\eps-\eps_0)}+\kappa^{m_0(1+\eps)}|\Omega|.
\end{align*}which is controlled by a constant depending on $d$, $\Lambda$ and the inner and outer radii of $\Omega$ and $\Omega^*$.
\end{proof}

%%%%%%%%%%%%%%%%%%%%%%%%%%%%%%%%%%%%%%%%%%%%%%%%%%%%%%%%%%%%%%%%%%%%%%%%%%%%%%%%%%%%%%%%%%%%%%%%%%%%%%%%


\begin{thebibliography}{100}


\bibitem[B]{B} Y. Brenier, {\em D\'ecomposition polaire et r\'earrangement monotone des champs de vecteurs}, C. R. Acad. Sci. Paris S\'er. I. Math. 305 (1987), 805-808. 


\bibitem[C1]{C1} L. Caffarelli, {\em The regularity of mappings with a convex potential}, J. Amer. Math. Soc. 5 (1992), no. 1, 99-104.
\bibitem[C2]{C2} L. Caffarelli, {\em Boundary regularity of maps with convex potentials}, Comm. Pure Appl. Math. 45 (1992), no. 9, 1141-1151.
\bibitem[C3]{C3} L. Caffarelli, {\em Boundary regularity of maps with convex potentials II}, Ann. of Math. (2) 144 (1996), no. 3, 453-496.

\bibitem[CG]{CG} L. Caffarelli, C. Guti\'errez, {\em Real analysis related to the \MA equation}, Trans. Amer. Math. Soc. 348 (1996), no. 3, 1075-1092. 

\bibitem[CLW]{CLW} S. Chen, J. Liu, X.J. Wang, {\em Global regularity for the Monge-Amp\`ere equation with natural boundary condition}, eprint arXiv:1802.07518.

\bibitem[CLW2]{CLW2} S. Chen, J. Liu, X.J. Wang, {\em Boundary regularity for the second boundary-value problem of Monge-Amp\`ere equations in dimension two}, eprint arXiv:1806.09482.

\bibitem[DF]{DF}G. De Philippis, A. Figalli, {\em $W^{2,1}$ regularity for solutions of the Monge-Amp\`ere equation}, Invent. Math. 192 (2013), no. 1, 55-69.
\bibitem[DFS]{DFS}G. De Philippis, A. Figalli, O. Savin, {\em A note on interior $W^{2,1+\eps}$ estimates for the \MA equation}, Math. Ann. 357 (2013), no. 1, 11-22.

\bibitem[F1]{F1} A. Figalli, {\em The \MA equation and its applications}, Zurich Lectures in Advanced Mathematics. European Mathematical Society, Z\"urich, 2017. 
\bibitem[F2]{F2} A. Figalli, {\em Global existence for the semigeostrophic equations via Sobolev estimates for \MA}, Partial Differential Equations and Geometric Measure Theory, 1-42, Lecture Notes in Math., CIME, Springer, Cham, 2018.


\bibitem[G]{G} C. Guti\'errez, {\em The \MA equation}, Progress in Nonlinear Differential Equations and their Applications, 89. Birkh\"auser/ Springer, Cham, 2016. 


\bibitem[J]{J} F. John, {\em Extremum problems with inequalities as subsidiary conditions}, Studies and Essays Presented to R. Courant on his 60th Birthday, Interscience, New York (1948), 187-204.

\bibitem[SY]{SY} O. Savin, H. Yu, {\em Regularity of optimal transport between planar convex domains}, eprint arXiv:1806.06252.

\bibitem[Sch]{Sch} T. Schmidt, {\em $W^{2,1+\eps}$ estimates for the \MA equation}, Adv. Math. 240 (2013), 672-689.



\bibitem[U]{U} J. Urbas, {\em On the second boundary value problem for equations of \MA type}, J. Reine Angew. Math. 487 (1997), 115-124. 
 
\bibitem[V]{V} C. Villani, {\em Topics in optimal transportation}, Graduate Studies in Mathematics, 58. American Mathematical Society, Providenc, RI, 2003.

\bibitem[W]{W} X.J. Wang,  {\em Some counterexamples to the regularity of Monge-Amp\`ere equations}, Proc. Amer. Math. Soc. 123 (3), 841-845.

\end{thebibliography}
\end{document}